\documentclass[twoside,leqno,10pt]{amsart}
\usepackage{amsfonts}
\usepackage{amsmath}
\usepackage{amscd}
\usepackage{amssymb}
\usepackage{amsthm}
\usepackage{latexsym}
\usepackage{color}
\newtheorem{theorem}{Theorem}
\newtheorem{Proposition}{Proposition}

\newtheorem{corollary}{Corollary}

\numberwithin{equation}{section}

\title{Asymptotic behavior of small solutions of quadratic congruences in three variables modulo prime powers}
\author{Stephan Baier \and Anup Haldar} 

\address{Stephan Baier,
Ramakrishna Mission Vivekananda Educational and Research Institute, Department of Mathematics, G. T. Road, PO Belur Math, Howrah, West Bengal 711202, India}
\email{stephanbaier2017@gmail.com}
\address{Anup Haldar,
Ramakrishna Mission Vivekananda Educational and Research Institute, Department of Mathematics, G. T. Road, PO Belur Math, Howrah, West Bengal 711202, India}
\email{anuphaldar1996@gmail.com}

\begin{document}
\begin{abstract}
Let $p>5$ be a fixed prime and assume that $\alpha_1,\alpha_2,\alpha_3$ are coprime to $p$. We study the asymptotic behavior of small solutions of congruences of the form $\alpha_1x_1^2+\alpha_2x_2^2+\alpha_3x_3^2\equiv 0\bmod{q}$ with $q=p^n$, where $\max\{|x_1|,|x_2|,|x_3|\}\le N$ and $(x_1x_2x_3,p)=1$. (In fact, we consider a smoothed version of this problem.) If $\alpha_1,\alpha_2,\alpha_3$ are fixed and $n\rightarrow \infty$, we establish an asymptotic formula (and thereby the existence of such solutions) under the condition $N\gg q^{1/2+\varepsilon}$. If these coefficients are allowed to vary with $n$, we show that this formula holds if $N\gg q^{11/18+\varepsilon}$. The latter should be compared with a result by Heath-Brown who established the  existence of non-zero solutions under the condition $N \gg q^{5/8+\varepsilon}$ for odd square-free moduli $q$. 
\end{abstract}
\maketitle
\tableofcontents

\section{Introduction and main results} 
Recently, the authors published a short article \cite{arx} titled ``Pythagorean triples modulo prime powers'' on the arXiv preprint server. In this article, we studied small solutions of quadratic congruences of the form 
$$
x_1^2+x_2^2-x_3^2\equiv 0\bmod{q},
$$ 
where $q=p^n$ is a power of a fixed prime $p$ and $n\rightarrow \infty$.
In the present paper, we investigate, more generally,
small solutions of quadratic congruences of the form
\begin{equation} \label{pythpn}
\alpha_1x_1^2+\alpha_2x_2^2+\alpha_3x_3^2\equiv 0\bmod{q}
\end{equation}
with prime power moduli $q$.
First, we give a brief review of some history of this problem. 

Let $Q(x_1,...,x_n)$ be a quadratic form with integer coefficients. The question of detecting small solutions of congruences of the form
$$
Q(x_1,...,x_k)\equiv 0 \bmod{q}
$$
has received a lot of attention (see, in particular, \cite{Hea1}, \cite{Hea2} and \cite{Hea3}). Here we focus on the case $k=3$. In this case, a result by Schinzel, Schlickewei and Schmidt \cite{SSS} for general moduli $q$ implies that there is a non-zero solution $(x_1,x_2,x_3)\in \mathbb{Z}^3$ such that $\max\{|x_1|,|x_2|,|x_3|\}=O(q^{2/3})$, where the $O$-constant is absolute. The exponent $2/3$ was improved to $5/8+\varepsilon$ by Heath-Brown \cite{Hea3} for forms with $(\det Q,q)=1$ and $q$ odd and square-free. 
A result by Cochrane \cite{Coc} for general moduli $q$ implies that for any {\it fixed} form $Q(x)$, there is a non-zero solution with $\max\{|x_1|,|x_2|,|x_3|\}=O(q^{1/2})$, where the $O$-constant may depend on the form. Of particular interest is the case when $q=p^n$ is a prime power. This was considered by Hakimi in \cite{Hak}, with emphasis on quadratic forms with a large number $k$ of variables. 

In the present paper, we study the {\it asymptotic behavior} of small solutions of diagonal quadratic congruences \eqref{pythpn}
with $\max\{|x_1|,|x_2|,|x_3|\}\le N$ and $(x_1x_2x_3,q)=1$ if $q=p^n$ is a power of a fixed odd prime $p$. The condition $(x_1x_2x_3,q)=1$ automatically excludes the trivial solution $(0,0,0)$. We also assume that $(\alpha_1\alpha_2\alpha_3,q)=1$. For convenience, we consider a smoothed version of this problem (i.e., the solutions are suitably weighted). We study this problem both for fixed and arbitrary coefficients $\alpha_i$. In the case of fixed coefficients $\alpha_i$ and $n\rightarrow\infty$, we obtain an asymptotic formula if $N\gg q^{1/2+\varepsilon}$, and in the case of coefficients which are allowed to vary with $n$, we obtain such a formula for $N\gg q^{11/18+\varepsilon}$. This should be compared with Heath-Brown's above-mentioned result from \cite{Hea3}. He focused on the ``orthogonal'' situation when $q$ is an odd square-free number and obtained the slightly weaker exponent $5/8=0.625$ in place of $11/18=0.6\overline{1}$, only addressing the {\it existence} of non-zero solutions. As pointed out by Heath-Brown in \cite{Hea3}, the existence of {\it non-zero} solutions $\ll q^{\theta}$ for {\it all} odd moduli follows if one has established it for all {\it square-free} odd moduli, using the following simple observation: If $q=q_0^2q_1$ with $q_1$ square-free and $Q(x_1,x_2,x_3)\equiv 0 \bmod{q_1}$, then $Q(q_0x_1,q_0x_2,q_0x_3)\equiv 0 \bmod{q}$. For powers $q=p^n$ of a fixed odd prime $p$, this argument even gives the existence of {\it non-zero} solutions $\ll_{\varepsilon} q^{1/2+\varepsilon}$ if $n$ is large enough (providing only a small fraction of all solutions). More precisely, we get a non-zero solution $\ll q^{1/2+1/(2n)}$. However, if we restrict ourselves to solutions satisfying $(x_1x_2x_3,q)=1$, then this argument does not work any longer since $q_0$ is itself a power of $p$ if $q=p^n$.  

Before we state our results, we explain why the exponent $1/2$ is the limit of our method by looking at the case $\alpha_1=1,\alpha_2=1,\alpha_3=-1$ of Pythagorean triples modulo prime powers (see also the discussion in \cite{arx}).
Small solutions of the congruence
\begin{equation} \label{pythcongruence}
x_1^2+x_2^2-x_3^2\equiv 0 \bmod{p^n}
\end{equation}
arise immediately from Pythagorean triples in $\mathbb{Z}^3$, provided that $p>5$. (If $p=2,3,5$, we have $x_1^2+x_2^2-x_3^2\not\equiv 0\bmod{p}$ if $(x_1x_2x_3,p)=1$.) It is known that the number of Pythagorean triples $(x_1,x_2,x_3)\in \mathbb{Z}^3$ satisfying $x_1^2+x_2^2=x_3^2$ such that $|x_3|\le N$ is $\sim cN\log N$ with $c=\pi/4$. It should not be difficult to modify this into an asymptotic of the form $\sim c_pN\log N$ with $c_p$ depending on $p$ if one includes the restriction $(x_1x_2x_3,p)=1$. If $N\le \sqrt{q/2}$ with $q=p^n$, then any solution $(x_1,x_2,x_3)$ of the congruence \eqref{pythpn} is in fact a Pythagorean triple. Hence, in this case, one expects an asymptotic of the form $\sim c_pN\log N$ for the number of solutions satisfying $(x_1x_2x_3,p)=1$ and $\max\{|x_1|,|x_2|,|x_3|\}\le N$ of the said congruence. In contrast, for much larger $N$, the expected number of solutions should be $\sim C_pN^3/q$ for a suitable constant $C_p>0$. In particular, one may expect this to hold for $N\ge q^{1/2+\varepsilon}$. Hence, there should be a transition between two different asymptotic formulas around the point $N=q^{1/2}$. Indeed, we will work out an asymptotic of the said form $\sim C_p\cdot N^3/q$ for $N\ge q^{\nu}$ with $\nu>1/2$ for general congruences of the form in \eqref{pythpn}, where $C_p$ depends on $p$ and the coefficients $\alpha_i$.  

To state our results, we define the following quantity
$$
C_p(\alpha_1,\alpha_2,\alpha_3):=\frac{(p-s_p(\alpha_1,\alpha_2,\alpha_3))(p-1)}{p^2},
$$
where
\begin{equation} \label{sdef}
s_p(\alpha_1,\alpha_2,\alpha_3):=2+\left(\frac{-\alpha_1\alpha_2}{p}\right)+
\left(\frac{-\alpha_1\alpha_3}{p}\right)+\left(\frac{-\alpha_2\alpha_3}{p}\right).
\end{equation}
We note that if $q=p$ is a prime, then the total number $(x_1,x_2,x_3)$ of solutions to the congruence \eqref{pythpn} satisfying $(x_1x_2x_3,p)=1$ turns out to be $(p-1)(p-s_p(\alpha_1,\alpha_2,\alpha_3))$.  So a solution exists if $p>s_p(\alpha_1,\alpha_2,\alpha_3)$. Our first  main result is as follows. 
 
\begin{theorem} \label{mainresult}
Let $\varepsilon>0$ be fixed, $p>2$ be a fixed prime and $\alpha_1,\alpha_2,\alpha_3$ be fixed integers which are coprime to $p$. Let $\Phi:\mathbb{R}\rightarrow \mathbb{R}_{\ge 0}$ be a Schwartz class function. Set $q:=p^n$. Then as $n\rightarrow \infty$, we have the asymptotic formula
\begin{equation} \label{main}
\sum\limits_{\substack{(x_1,x_2,x_3)\in \mathbb{Z}^3\\ (x_1x_2x_3,p)=1\\ \alpha_1x_1^2+\alpha_2x_2^2+\alpha_3x_3^2 \equiv 0 \bmod{q}}} \Phi\left(\frac{x_1}{N}\right)
\Phi\left(\frac{x_2}{N}\right)\Phi\left(\frac{x_3}{N}\right)\sim
\hat{\Phi}(0)^3\cdot C_p(\alpha_1,\alpha_2,\alpha_3)\cdot \frac{N^3}{q},
\end{equation}
provided that $N\ge q^{1/2+\varepsilon}$ and $p>s_p(\alpha_1,\alpha_2,\alpha_3)$.
\end{theorem}

Secondly, we establish the following result. 

\begin{theorem} \label{mainresult2}
Let the conditions in Theorem \ref{mainresult} be kept except that $\alpha_1,\alpha_2,\alpha_3$ are no longer fixed but allowed to vary with $n$. Also suppose that $p>s_p(\alpha_1,\alpha_2,\alpha_3)$. Then the asymptotic formula \eqref{main} holds if $N\ge q^{11/18+\varepsilon}$.
\end{theorem}

Using the rapid decay of the weight function $\Phi$, we obtain the following existence result as a corollary of Theorems \ref{mainresult} and \ref{mainresult2} above. 

\begin{corollary} \label{mainresult3} Let $\varepsilon>0$ be fixed and $p>2$ be a fixed prime.  Set $q:=p^n$. For $\alpha_1,\alpha_2,\alpha_3$ being integers such that the congruence $\alpha_1x_1^2+\alpha_2x_2^2+\alpha_3x_3^3\equiv 0\bmod q$ is solvable in integers with $p\not| x_1x_2x_3$, let $m(\alpha_1,\alpha_2,\alpha_3;q)$ be the smallest value of $\max\{|x_1|,|x_2|,|x_3|\}$ for such a solution. If no such solution exists, set 
$m(\alpha_1,\alpha_2,\alpha_3;q)=0$. Then we have the following. \medskip\\
(i) If $\alpha_1,\alpha_2,\alpha_3$ are fixed and satisfy $(\alpha_1\alpha_2\alpha_3,p)=1$, then, as $n\rightarrow\infty$,
\begin{equation} \label{parti}
m(\alpha_1,\alpha_2,\alpha_3;q)\ll q^{1/2+\varepsilon},
\end{equation}
where the implied constant depends only on $p$, $\alpha_1,\alpha_2,\alpha_3$ and $\varepsilon$. \medskip\\
(ii) As $n\rightarrow\infty$, 
\begin{equation} \label{partii}
\max\limits_{\substack{\alpha_1,\alpha_2,\alpha_3 \bmod{q}\\ (\alpha_1\alpha_2\alpha_3,p)=1}} m(\alpha_1,\alpha_2,\alpha_3;q)\ll q^{11/18+\varepsilon},
\end{equation}
where the implied constant depends only on $p$ and $\varepsilon$. 
\end{corollary}
$ $\\
{\bf Comments:}\medskip\\ 
(a) Legendre \cite{Leg} gave a criterion for the non-trivial representability of 0 by a diagonal ternary quadratic form. If, in particular, there exists a solution $(x_1,x_2,x_3)$ of the equation $Q(x_1,x_2,x_3)=0$ with $(x_1x_2x_3,p)=1$, then \eqref{parti} holds trivially. Thus, part (i) of Corollory \ref{mainresult} is of interest only if $Q(x_1,x_2,x_3)$ does not represent 0 in the above form. Part (ii) of Corollary \ref{mainresult} is of general interest.\medskip\\
(b) With some extra effort, it is possible to sharpen the bound in part (i) of the above Corollary \ref{mainresult3} to
$$
m(\alpha_1,\alpha_2,\alpha_3;q)\le C(p,\varepsilon)\max\{|\alpha_1|,|\alpha_2|,|\alpha_3|\}q^{1/2+\varepsilon},
$$ 
where the function $C$ depends on the weight function $\Phi$. Similarly, the implied constant in part (ii) is of the form $D(p,\varepsilon)$, where the function $D$ depends on $\Phi$. \\
 
We begin with proving Theorem \ref{mainresult} in two parts. In the first part we deal with the case when one of $-\alpha_i\alpha_j$ with $i\not= j$ is a quadratic residue modulo $p$ (without loss of generality, we may take $i=2$ and $j=3$). In the second part we cover the complementary case when none of $-\alpha_i\alpha_j$ with $i\not=j$ is a quadratic residue modulo $p$.
Key ingredients in our method are a parametrization of $\mathbb{Q}_p$-rational points $(z_1,z_2)$ on the conic
$$
\alpha_1z_1^2+\alpha_2z_2^2=-\alpha_3,
$$
repeated use of Poisson summation and an explicit evaluation of complete exponential sums with rational functions to prime power moduli due to Cochrane \cite{CoZ}. This transforms the problem into a dual problem which in the case of fixed coefficients amounts to counting solutions of quadratic Diophantine {\it equations} (rather than {\it congruences}). Our method generalizes that in \cite{arx}. 

To prove Theorem \ref{mainresult2}, we will observe that here our dual problem essentially amounts to bounding from above the number of solutions of {\it congruences} of the form
$$
\beta_1x_1^2+\beta_2x_2^2+\beta_3x_3^2\equiv 0 \bmod{q'}
$$
with $q'|q$ in boxes $\max\{|x_1|,|x_2|,|x_3|\}\le M$ with $M$ of size roughly $q'/N$,
where $\beta_1,\beta_2,\beta_3$ depend on $\alpha_1,\alpha_2,\alpha_3$. So these new boxes are much smaller than those in the original problem, but we don't need to establish an asymptotic formula here. To obtain an upper bound which allows us to beat the exponent $2/3$, we consider two cases depending on Diophantine properties of the fractions $\beta_1\overline{\beta_3}/q'$ and $\beta_2\overline{\beta_3}/q'$, where $\overline{\beta_3}$ is a multiplicative inverse of $\beta_3$ modulo $q'$. In the first case, we shall reduce the problem using the Cauchy-Schwarz inequality to counting small solutions of certain linear congruences. In the second case we turn the problem into counting solutions of quadratic Diophantine {\it equations}, similarly as in our proof of Theorem \ref{mainresult}.

Instead of results with smooth weights, it should also be possible to produce similarly strong results with sharp cutoff, but the technical details become then more complicated.\\ \\
{\bf Acknowledgements.} The authors would like to thank the Ramakrishna Mission Vivekananda Educational and Research Institute for providing excellent working conditions. The second-named author would like to thank CSIR, Govt. of India for financial support in the form of a Junior Research Fellowship under file number 09/934(0016)/2019-EMR-I. \\ \\
{\bf Data availability statement:} This manuscript has no associated data.\\ \\
{\bf Conflict of interest statement:} The authors have no conflicts of interest to declare. All co-authors have seen and agree with the contents of the manuscript and there is no financial interest to report. We certify that the submission is original work and is not under review at any other publication.

\section{Preliminaries}
We will use the notation 
$$
e_q(z):=e\left(\frac{z}{q}\right)=e^{2\pi i z/q}
$$
for $q\in \mathbb{N}$ and denote by $G_q$ the quadratic Gauss sum
$$
G_q=\sum\limits_{x=1}^q e_q(x^2).
$$
We recall that if $q$ is odd, then 
$$
G_q=\sum\limits_{y=1}^q \left(\frac{y}{q}\right)e_q(y),
$$
where $\left(\frac{y}{q}\right)$ is the Jacobi symbol. We further recall that in this case, $|G_q|=\sqrt{q}$. Throughout the sequel, we will write $\overline{\alpha}$ for a multiplicative inverse of $\alpha$ to the relevant modulus, which will always be apparent from the context. Implicitly, we often make use of the fact that $\overline{\alpha}$ is a quadratic (non-)residue modulo $p$ if and only if $\alpha$ is. 

The following preliminaries will be needed in the course of this paper.

\begin{Proposition}[Parametrization of points on a conic] \label{para}
Let $K$ be a field and assume that $\alpha_1,\alpha_2,\alpha_3\in K^{\ast}$ and $a,b\in K$ such that 
$$
\alpha_1a^2+\alpha_2b^2=-\alpha_3.
$$
Then all $K$-rational points on the conic
\begin{equation*} 
\alpha_1z_1^2+\alpha_2z_2^2=-\alpha_3
\end{equation*}
are parametrized in the form
$$
(z_1,z_2)=\left(a-2\alpha_2\frac{at^2-bt}{\alpha_1+\alpha_2t^2},-b-2\alpha_1\frac{at-b}{\alpha_1+\alpha_2t^2}\right),
$$
where $t\in K$ with $t^2\not=-\alpha_1/\alpha_2$. The map 
\begin{equation} \label{bijection}
m : \left\{t\in K : t^2\not=-\frac{\alpha_1}{\alpha_2}\right\} \longrightarrow \left\{(z_1,z_2)\in K^2 : \alpha_1z_1^2+\alpha_2z_2^2=-\alpha_3\right\}
\end{equation}
defined by 
\begin{equation} \label{mt}
m(t):=\left(a-2\alpha_2\frac{at^2-bt}{\alpha_1+\alpha_2t^2},-b-2\alpha_1\frac{at-b}{\alpha_1+\alpha_2t^2}\right)
\end{equation}
is bijective. 
\end{Proposition}

\begin{proof}
We use a standard method of parametrization. Given a $K$-rational point $Q=(a,b)$ on the conic and $t\in K$, by B\'ezout's theorem, the line $\mathcal{L}(t)$ through $Q$ given by the equation $z_2-b=t(z_1-a)$ intersects the conic in $Q$ and at most one more point $m(t)$. This point may be $Q$ itself, in which case $\mathcal{L}(t)$ is the tangent to the conic at $Q$. Conversely, if $P$ is a $K$-rational point on the conic, then there exists precisely one line through $P$ and $Q$ (which is the tangent in the case $P=Q$) with rational slope $t$. Hence, we have a bijection between the $K$-rational points on the conic and the set of $t\in K$ for which $m(t)$ exists. To find $m(t)$, we write 
$$\alpha_1z_1^2+\alpha_2z_2^2=\alpha_1a^2+\alpha_2b^2,$$
which implies
$$\alpha_1(z_1-a)(z_1+a)=\alpha_2(z_2-b)(z_2+b).$$
Now we plug in $z_2-b=t(z_1-a)$ and get
$$\alpha_1(z_1+a)+\alpha_2t(2b+t(z_1-a))=0.$$
Therefore 
\begin{equation*}
\begin{split}
z_1=&\frac{-a\alpha_1-\alpha_2t(2b-at)}{\alpha_1+\alpha_2t^2}\\=&-a+2\alpha_2t\frac{at-b}{\alpha_1+\alpha_2t^2}  
\end{split}
\end{equation*}
and 
\begin{equation}
    \begin{split}
        z_2&=b+(z_1-a)t\\
        &=b+t\left(-2a+2\alpha_2t\frac{at-b}{\alpha_1+\alpha_2t^2}\right)\\&=b-2t\frac{a\alpha_1+b\alpha_2t}{\alpha_1+\alpha_2t^2}\\&=-b-2\alpha_1\frac{at-b}{\alpha_1+\alpha_2t^2},
    \end{split}
\end{equation}
which exist if $t^2\not= -\alpha_1/\alpha_2$.
This gives the desired parametrization in \eqref{mt}, and the map in \eqref{bijection} is bijective. 
\end{proof}

\begin{Proposition} \label{nofroot}
Let $p>2$ be a prime and $n\in \mathbb{N}$. Then the number of solutions $(x_1,x_2) \bmod p^n$ of the congruence 
$$
\gamma_1x_1^2+\gamma_2x_2^2\equiv 1\bmod{p^n}
$$ 
is $p^n+p^{n-1}$ if $(\gamma_1\gamma_2,p)=1$ and $-\gamma_1\gamma_2$ is a quadratic non-residue modulo $p$.
\end{Proposition}

\begin{proof}
This is given in \cite[Corollary 35]{NoS}.
\end{proof}

\begin{Proposition} \label{quadequations} Let $A,B,C\in \mathbb{Z}\setminus \{0\}$ and $x\ge 1$. Then the number of solutions $(X,Y)\in \mathbb{Z}$ with $\max\{|X|,|Y|\}\le x$ of the equation
$$
AX^2+BY^2=C
$$ 
is bounded by $O(|ABCx|^{\varepsilon})$. 
\end{Proposition}

\begin{proof} We first multiply the equation in question by $A$ to get
$$
(AX)^2+ABY^2=AC.
$$
Hence, it suffices to prove that there are at most $O(|ABCx|^{\varepsilon})$ solutions of the equation
$$
U^2+ABV^2=AC
$$
with $\max\{|U|,|V|\}\le |A|x$. We may write the above equation as 
$$
(U+V\sqrt{-AB})(U-V\sqrt{-AB})=AC.
$$ 
Let $\mathcal{O}_K$ be the ring of integers of the number field $K:=\mathbb{Q}(\sqrt{-AB})$. The number of divisors of the ideal $\mathfrak{C}=(AC)$ is $\ll \mathcal{N}_{K:\mathbb{Q}}(\mathfrak{C})^{\varepsilon}\le |AC|^{2\varepsilon}$. Hence, if $\mathfrak{A}$ is a principal ideal divisor of $\mathfrak{C}$, then it suffices to show that $\mathfrak{A}$ has at most $O(|ABCx|^{\varepsilon})$ generators of the form $U+V\sqrt{-AB}$ with $\max{|U|,|V|}\le |A|x$. Hence, we must show that if $U_0+V_0\sqrt{-AB}$ is such a generator, then there are at most $O(|ABCx|^{\varepsilon})$ units in $\mathcal{O}_K$ such that $u(U_0+V_0\sqrt{-AB})=U_1+V_1\sqrt{-AB}$, where 
$\max\{|U_1|,|V_1|\}\le |A|x$. This is trivial if $-AB$ is a square or $AB>0$ since then the number of units in $\mathcal{O}_K$ is bounded by 6.  Assume now that $AB<0$ is not a square and $-AB=st^2$, where $s$ is square-free. Then $K=\mathbb{Q}(\sqrt{s})$ is a real-quadratic field, and there is a fundamental unit of the form 
$$
\epsilon=a+b\sqrt{s}>1
$$ 
with $a,b\in \mathbb{Z}/2$, where $a,b\not=0$. Hence, $u=\pm \epsilon^k$ for some $k\in \mathbb{Z}$. Let $R:=|A|^{3/2}|B|^{1/2}|x|$. We observe that
$$
\frac{1}{R}\ll U_1+V_1\sqrt{-AB}\ll R
$$
since 
$$
|U_1+V_1\sqrt{-AB}|\cdot |U_1-V_1\sqrt{-AB}|=|AC|\ge 1.  
$$
Hence,
$$
\frac{1}{R}\ll \epsilon^k|U_0+V_0\sqrt{-AB}| \ll R,
$$
and the number of possible $k$'s is bounded by 
$$
\ll \log_{\epsilon} R^2\ll \log R \ll \log|2ABx|, 
$$
using the well-known fact that $\epsilon\ge (1+\sqrt{5})/2$ for every $s$. 
Thus, we get $O(\log|2ABx|)$ possible units $u$, which completes the proof. 
\end{proof}

\begin{Proposition}[Poisson summation formula] \label{Poisson} Let $\Phi : \mathbb{R}\rightarrow \mathbb{R}$ be a Schwartz class function, $\hat\Phi$ its Fourier transform. Then
$$
\sum\limits_{n\in \mathbb{Z}} \Phi(n)=\sum\limits_{n\in \mathbb{Z}} \hat\Phi(n).
$$
\end{Proposition}

\begin{proof}
See \cite[section 3]{notes}. 
\end{proof}

\begin{Proposition}[Evaluation of exponential sums with rational functions] \label{Expsums}
Let $p>2$ be a prime, $n\ge 2$ be a natural number and $f=F_1/F_2$ be a rational function where $F_1,F_2\in \mathbb{Z}[x]$. For a polynomial $G$ over $\mathbb{Z}$, let $\mbox{ord}_p(G)$ be the largest power of $p$ dividing all of the coefficients of $G$, and for a rational function $g=G_1/G_2$ with $G_1$ and $G_2$ polynomials over $\mathbb{Z}$, let $\mbox{ord}_p(g) := \mbox{ord}_p(G_1)-\mbox{ord}_p(G_2)$. Set
$$
r:=\mbox{ord}_p(f'),
$$
and 
$$
S_{\alpha}(f;p^n):=\sum\limits_{\substack{x=1\\ x\equiv \alpha \bmod{p}}}^{p^n} e_{p^n}(f(x)), 
$$
where $\alpha\in \mathbb{Z}$.
Then we have the following if $r\le n-2$ and $(F_2(\alpha),p)=1$.\medskip\\
(i) If $p^{-r}f'(\alpha)\not\equiv 0\bmod{p}$, then $S_{\alpha}(f,p^n) = 0$.\medskip\\
(ii) If $\alpha$ is a root of $p^{-r}f'(x)\equiv 0\bmod{p}$ of multiplicity one, then
$$
S_{\alpha}(f;p^n) =\begin{cases} e\left(f(\alpha^{\ast})\right)p^{(n+r)/2} & \mbox{ if } n-r \mbox{ is even,}\\
e\left(f(\alpha^{\ast})\right)p^{(n+r)/2}\left(\frac{A(\alpha)}{p}\right)\cdot \frac{G_p}{\sqrt{p}} & \mbox{ if } n-r \mbox{ is odd,}
\end{cases}
$$
where $\alpha^{\ast}$ is the unique lifting of $\alpha$ to a solution of the congruence $p^{-r}f'(x) \equiv 0 \bmod p^{[(n-r+1)/2]}$ and 
$$
A(\alpha):=2p^{-r}f''(\alpha^{\ast}).
$$
\end{Proposition}

\begin{proof} This is \cite[Theorem 3.1(iii)]{CoZ}.
\end{proof}

\section{Proof of Theorem \ref{mainresult}}

\subsection{Parametrization of solutions - Case I} 
In the following, we consider the case when $-\alpha_2\alpha_3$ is a quadratic residue modulo $p$.

The congruence in question resembles the equation 
\begin{equation} \label{conic}
\alpha_1 x_1^2+\alpha_2 x_2^2+\alpha_3x_3^2=0
\end{equation} 
of a conic in homogeneous coordinates. Let us first see that every solution $(x_1,x_2,x_3)\in \mathbb{Z}^3$ with $(x_1x_2x_3,p)=1$ of the congruence 
\begin{equation} \label{congru1}
\alpha_1 x_1^2+\alpha_2 x_2^2+\alpha_3x_3^2\equiv 0 \bmod{p^n}
\end{equation}
comes from a solution of \eqref{conic} in the $p$-adic integers. To this end, we need to use a Hensel-type argument. Let $(x_1,x_2,x_3)$ be such a solution. We would like to lift it to a solution $(\tilde{x}_1,\tilde{x}_2,\tilde{x}_3)$ of the congruence
\begin{equation} \label{congru2}
\alpha_1\tilde{x}_1^2+\alpha_2\tilde{x}_2^2+\alpha_3\tilde{x}_3^2\equiv 0\bmod{p^{n+1}}.
\end{equation} 
So we consider $\tilde{x}_i=x_i+k_ip^n$ with $k_i=0,...,p-1$, $i=1,2,3$ and satisfying
$$
\alpha_1(x_1+k_1p^n)^2+\alpha_2(x_2+k_2p^n)^2+\alpha_3(x_3+k_3p^n)^2\equiv 0 \bmod{p^{n+1}}.
$$ 
Expanding the squares and using $2n\ge n+1$, this is equivalent to 
$$
\alpha_1x_1^2+\alpha_2x_2^2+\alpha_3x_3^2+2\alpha_1k_1p^nx_1+2\alpha_2k_2p^nx_2+2\alpha_3k_3p^nx_3\equiv 0 \bmod{p^{n+1}},
$$
which in turn is equivalent to
$$
\frac{\alpha_1x_1^2+\alpha_2x_2^2+\alpha_3x_3^2}{p^n}+2\alpha_1x_1k_1+2\alpha_2x_2k_2+2\alpha_3x_3k_3\equiv 0\bmod{p}.
$$
This linear congruence in $k_1,k_2,k_ 3$ has exactly $p^2$ solutions. So in particular, a solution $(x_1,x_2,x_3)$ of \eqref{congru1} lifts to a solution $(\tilde{x}_1,\tilde{x}_2,\tilde{x}_3)$ of \eqref{congru2}. So indeed, every solution of \eqref{congru1} arises from a solution of \eqref{conic} in $\mathbb{Z}_p$. 

Now we parametrize these solutions. Since $-\alpha_2\alpha_3$ is a quadratic residue modulo $p$, there exists $b\in \mathbb{Q}_p$ such that $b^2=-\alpha_3/\alpha_2$. Then $(0,-b)$ is a point on the conic
\begin{equation} \label{conics}
\alpha_1z_1^2+\alpha_2z_2^2=-\alpha_3,
\end{equation}
and using Proposition \ref{para} with $(0,-b)$ in place of $(a,b)$, the $\mathbb{Q}_p$-rational points $(z_1,z_2)$ on this conic are parametrized as 
$$\left(z_1,z_2\right)=\left(-\frac{2tb\alpha_2}{\alpha_1+\alpha_2t^2},-\frac{b(\alpha_1-\alpha_2t^2)}{\alpha_1+\alpha_2 t^2}\right),
$$
where $t\in \mathbb{Q}_p$ with $t^2\not=-\alpha_1/\alpha_2$.    
Now if $(x_1,x_2,x_3)$ is a solution of \eqref{conic} in the $p$-adic integers, where $x_3\not=0$, then
$$
\left(\frac{x_1}{x_3},\frac{x_2}{x_3}\right)
$$
is a point on the conic in \eqref{conics}. Hence, we have 
$$
\left(\frac{x_1}{x_3},\frac{x_2}{x_3}\right)=\left(-\frac{2tb\alpha_2}{\alpha_1+\alpha_2t^2},-\frac{b(\alpha_1-\alpha_2t^2)}{\alpha_1+\alpha_2 t^2}\right).
$$
But we restricted ourselves to triples $(x_1,x_2,x_3)$ with $|x_i|_p=1$, $i=1,2,3$. Hence, we have 
$$
1=\left|\frac{x_1}{x_3}\right|_p=\left|\frac{2tb\alpha_2}{\alpha_1+\alpha_2t^2}\right|_p
$$
and 
$$
1=\left|\frac{x_2}{x_3}\right|_p=\left|\frac{b(\alpha_1-\alpha_2t^2)}{\alpha_1+\alpha_2 t^2}\right|_p.
$$
We claim that $|t|_p=1$. Indeed,
if $|t|_p<1$, then
$$
\left|\frac{2tb\alpha_2}{\alpha_1+\alpha_2t^2}\right|_p=|t|_p<1,
$$
a contradiction, and if $|t|_p>1$, then
$$
\left|\frac{2tb\alpha_2}{\alpha_1+\alpha_2t^2}\right|_p=|t|_p^{-1}<1.
$$
Hence $|t|_p=|\alpha_1+\alpha_2t^2|_p=|\alpha_1-\alpha_2t^2|_p=1$, and 
$$
x_1=-b(\alpha_1-\alpha_2t^2)u, \quad x_2=-2b\alpha_2tu, \quad x_3=(\alpha_1+\alpha_2t^2)u, 
$$
where $u$ is a unit in the ring of $p$-adic integers $\mathbb{Z}_p$, i.e. $|u|_p$=1. 

By reducing modulo $p^n$, we deduce that the solutions of the congruence
$$
\alpha_1y_1^2+\alpha_2y_2^2+\alpha_3\equiv 0 \bmod{p^n}
$$
with $(y_1y_2,p)=1$ are parametrized in the form
\begin{equation} \label{para1}
y_1=-\frac{b(\alpha_1-\alpha_2t^2)}{\alpha_1+\alpha_2t^2}, \quad y_2=-\frac{2b\alpha_2t}{\alpha_1+\alpha_2t^2}, \quad t \bmod{p^n},\ (t(\alpha_1-\alpha_2t^2)(\alpha_1+\alpha_2t^2),p)=1.
\end{equation} 
Here $1/(\alpha_1+\alpha_2t^2)$ stands for a multiplicative inverse of $\alpha_1+\alpha_2t^2 \bmod{p^n}$. Moreover, the pairs $(y_1,y_2)$ given as in \eqref{para1} are distinct modulo $p^n$ by the following argument: Suppose that 
\begin{equation} \label{t1t21}
\frac{b(\alpha_1-\alpha_2t_1^2)}{\alpha_1+\alpha_2t_1^2}\equiv \frac{b(\alpha_1-\alpha_2t_2^2)}{\alpha_1+\alpha_2t_2^2} \bmod{p^n}
\end{equation}
and 
\begin{equation} \label{t1t22}
\frac{2b\alpha_2t_1}{\alpha_1+\alpha_2t_1^2}\equiv \frac{2b\alpha_2t_2}{\alpha_1+\alpha_2t_2^2} \bmod{p^n}.
\end{equation}
Then from \eqref{t1t21} it follows upon multiplying both sides with the denominators that 
$$
b(\alpha_1^2-\alpha_2^2t_1^2t_2^2+\alpha_1\alpha_2t_2^2-\alpha_1\alpha_2t_1^2)\equiv b(\alpha_1^2-\alpha_2^2t_1^2t_2^2-\alpha_1\alpha_2t_2^2+\alpha_1\alpha_2t_1^2)\bmod{p^{n}}
$$
and hence 
$$
t_1^2\equiv t_2^2\bmod{p^n}.
$$
So if $t_1\not\equiv t_2\bmod{p^n}$, then $t_1\equiv -t_2\bmod{p^n}$. However, in this case, \eqref{t1t22} implies $t_1\equiv t_2\equiv 0\bmod{p^n}$ contradicting the assumption that $(t_i,p)=1$ for $i=1,2$. 

\subsection{Double Poisson summation}\label{fdob} We start by writing 
\begin{equation*}
\begin{split}
T= & \sum\limits_{\substack{(x_1,x_2,x_3)\in \mathbb{Z}^3\\ (x_1x_2x_3,p)=1\\ \alpha_1x_1^2+\alpha_2x_2^2+\alpha_3x_3^2\equiv 0 \bmod{p^n}}} \Phi\left(\frac{x_1}{N}\right)\Phi\left(\frac{x_2}{N}\right)\Phi\left(\frac{x_3}{N}\right)\\
=& \sum\limits_{(x_3,p)=1} \Phi\left(\frac{x_3}{N}\right)\sum\limits_{\substack{y_1,y_2\bmod{p^n}\\ (y_1y_2,p)=1 \\ \alpha_1y_1^2+\alpha_2y_2^2+\alpha_3\equiv 0\bmod{p^n}}} \sum\limits_{\substack{x_1\equiv x_3y_1\bmod{p^n}\\ x_2\equiv x_3y_2\bmod{p^n}}} \Phi\left(\frac{x_1}{N}\right)\Phi\left(\frac{x_2}{N}\right).
\end{split}
\end{equation*}
Now we apply Poisson summation, Proposition \ref{Poisson}, after a linear change of variables to the inner double sum over $x_1$ and $x_2$, obtaining
\begin{equation*}
T= \frac{N^2}{p^{2n}}\sum\limits_{(x_3,p)=1} \Phi\left(\frac{x_3}{N}\right)\sum\limits_{(k_1,k_2)\in \mathbb{Z}^2} \hat{\Phi}\left(\frac{k_1N}{p^n}\right)\hat{\Phi}\left(\frac{k_2N}{p^n}\right)\sum\limits_{\substack{y_1,y_2\bmod{p^n}\\ (y_1y_2,p)=1 \\ \alpha_1y_1^2+\alpha_2y_2^2+\alpha_3\equiv 0\bmod{p^n}}} e_{p^n}\left(k_2x_3y_1+k_1x_3y_2\right).
\end{equation*}
Using our parametrization \eqref{para1}, we deduce that
\begin{equation*}
T= \frac{N^2}{p^{2n}}\sum\limits_{(x_3,p)=1} \Phi\left(\frac{x_3}{N}\right)\sum\limits_{(k_1,k_2)\in \mathbb{Z}^2} \hat{\Phi}\left(\frac{k_1N}{p^n}\right)\hat{\Phi}\left(\frac{k_2N}{p^n}\right)\sum\limits_{\substack{t \bmod{p^n}\\
(t(\alpha_1-\alpha_2t^2)(\alpha_1+\alpha_2t^2),p)=1 }} e_{p^n}\left(x_3\cdot \frac{2k_1b\alpha_2t+k_2b(\alpha_1-\alpha_2t^2)}{\alpha_1+\alpha_2t^2}\right).
\end{equation*}

We decompose $T$ into 
\begin{equation} \label{divide}
T=T_0+U,
\end{equation}
where $T_0$ is the main term contribution of $(k_1,k_2)=(0,0)$. It follows that
\begin{equation*}
T_0= \hat{\Phi}(0)^2\cdot \frac{N^2}{p^{2n}}\sum\limits_{(x_3,p)=1} \Phi\left(\frac{x_3}{N}\right) \cdot p^{n-1}(p-s_p(\alpha_1,\alpha_2,\alpha_3)),
\end{equation*}
where $s_p(\alpha_1,\alpha_2,\alpha_3))$ is defined as in \eqref{sdef}. To see this, we note that the congruence $(\alpha_1-\alpha_2t^2)(\alpha_1+\alpha_2t^2)\equiv 0 \bmod{p}$ has four solutions modulo $p$ if 
$$
\left(\frac{-\alpha_1\alpha_2}{p}\right)=1 \quad \mbox{and} \quad  
\left(\frac{\alpha_1\alpha_2}{p}\right)=1,
$$ 
two solution if these Legendre symbols have opposite signs and no solution if they are both equal to $-1$. Moreover, since we assumed that 
$$
\left(\frac{-\alpha_2\alpha_3}{p}\right)=1,
$$ 
we have 
$$
\left(\frac{\alpha_1\alpha_2}{p}\right)=\left(\frac{-\alpha_1\alpha_3}{p}\right).
$$
In each case, the congruence 
$t(\alpha_1-\alpha_2t^2)(\alpha_1+\alpha_2t^2)\equiv 0 \bmod{p}$
has precisely $s_p(\alpha_1,\alpha_2,\alpha_3)$ solutions. 

If $N\ge p^{n\varepsilon}$ for any fixed $\varepsilon>0$, then the term $T_0$ can be simplified into
\begin{equation} \label{T0I}
\begin{split}
T_0= & \hat{\Phi}(0)^2 \cdot \frac{p-s_p(\alpha_1,\alpha_2,\alpha_3)}{p}\cdot \frac{N^2}{p^n} \cdot \left(\sum\limits_{x} \Phi\left(\frac{x}{N}\right) -\sum\limits_{x} \Phi\left(\frac{x}{N/p}\right)\right)\\
= & \hat{\Phi}(0)^2 \cdot \frac{p-s_p(\alpha_1,\alpha_2,\alpha_3)}{p}\cdot \frac{N^2}{p^n}\cdot \left(N\cdot \frac{p-1}{p}\cdot \hat\Phi(0)
+\sum\limits_{y\in \mathbb{Z}\setminus\{0\}} \left(N\hat\Phi(Ny)-\frac{N}{p}\cdot \hat\Phi\left(\frac{Ny}{p}\right)\right)\right)\\
= & \hat{\Phi}(0)^3 \cdot \frac{(p-s_p(\alpha_1,\alpha_2,\alpha_3))(p-1)}{p^2}\cdot \frac{N^3}{p^{n}}\cdot\left(1+o(1)\right)=\hat{\Phi}(0)^3 \cdot C_p(\alpha_1,\alpha_2,\alpha_3)\cdot \frac{N^3}{p^{n}}\cdot\left(1+o(1)\right)
\end{split}
\end{equation}
as $n\rightarrow\infty$, 
where we again use Poisson summation for the sums over $x$ above and the rapid decay of $\hat\Phi$. 

\subsection{Evaluation of exponential sums} Now we look at the error contribution
\begin{equation} \label{errorcont}
U= \frac{N^2}{p^{2n}}\sum\limits_{(x_3,p)=1} \Phi\left(\frac{x_3}{N}\right)\sum\limits_{(k_1,k_2)\in \mathbb{Z}^2\setminus \{(0,0)\}} \hat{\Phi}\left(\frac{k_1N}{p^n}\right)\hat{\Phi}\left(\frac{k_2N}{p^n}\right) E\left(k_1,k_2,x_3;p^n\right)
\end{equation}
with 
\begin{equation*}
E\left(k_1,k_2,x_3;p^n\right):=\sum\limits_{\substack{t \bmod{p^n}\\
(t(\alpha_1-\alpha_2t^2)(\alpha_1+\alpha_2t^2),p)=1 }} e_{p^n}\left(x_3\cdot \frac{2k_1b\alpha_2t+k_2b(\alpha_1-\alpha_2t^2)}{\alpha_1+\alpha_2t^2}\right).
\end{equation*}
Assume that
$$
(k_1,k_2,p^n)=p^r.
$$
The contribution of $r=n-1,n$ to the right-hand side of \eqref{errorcont} is $O_{\varepsilon}(1)$ if $N\ge p^{n\varepsilon}$ by the rapid decay of $\hat\Phi$ since $(k_1,k_2)=(0,0)$ is excluded from the summation. In the following, we assume that $r\le n-2$ so that Proposition \ref{Expsums} is applicable.

We split the inner sum over $t$ into 
\begin{equation*} \label{split}
E\left(k_1,k_2,x_3;p^n\right)=\sum\limits_{\substack{\alpha=1\\ \alpha \not\equiv 0 \bmod{p}\\\alpha^2\not\equiv \pm \alpha_1\overline{\alpha_2}\bmod{p}}}^p S_{\alpha}\left(f_{k_1,k_2};p^n\right),
\end{equation*}
where 
\begin{equation*} \label{Salphadef}
S_{\alpha}\left(f_{k_1,k_2};p^n\right)=\sum\limits_{\substack{t \bmod{p^n}\\ t\equiv \alpha\bmod{p}}} e_{p^n}\left(f_{k_1,k_2}(t)\right)
\end{equation*}
with 
\begin{equation*} \label{fdef}
f_{k_1,k_2}(t):=x_3\cdot \frac{2k_1b\alpha_2t+k_2b(\alpha_1-\alpha_2t^2)}{\alpha_1+\alpha_2t^2}.
\end{equation*}
Here, for convenience, we have suppressed the dependency on $x_3$ in our notations of $S_{\alpha}\left(f_{k_1,k_2};p^n\right)$ and $f_{k_1,k_2}(t)$. 
We calculate that 
\begin{equation*}
\begin{split}
f_{k_1,k_2,x_3}'(t)
= & 2x_3b\alpha_2\cdot \frac{k_1(\alpha_1-\alpha_2t^2)-2k_2\alpha_1t}{(\alpha_1+\alpha_2t^2)^2}.
\end{split}
\end{equation*}
Set
$$
l_1:=\frac{k_1}{p^r}, \quad l_2:=\frac{k_2}{p^r}.
$$
Then using Proposition \ref{Expsums}, if $\alpha_2\alpha^2+\alpha_1\not\equiv 0\bmod{p}$, we have $S_{\alpha}\left(f_{k_1,k_2,x_3};p^n\right)=0$ unless 
\begin{equation} \label{keycong}
2\alpha_1l_2\alpha\equiv l_1(\alpha_1-\alpha_2\alpha^2) \bmod{p}.
\end{equation}
If $\alpha^2\not\equiv 0,\alpha_1\overline{\alpha_2}\bmod{p}$, then it follows that $(l_1l_2,p)=1$ and 
$$
(k_1,p^n)=p^r=(k_2,p^n).
$$
In summary, we have 
\begin{equation*}
U= \frac{N^2}{p^{2n}}\sum\limits_{(x_3,p)=1} \Phi\left(\frac{x_3}{N}\right)\sum\limits_{r=0}^{n-2} \sum\limits_{\substack{\alpha=1\\ \alpha^2 \not\equiv 0,\pm \alpha_1\overline{\alpha_2}
\bmod{p}}}^p \sum\limits_{\substack{(l_1,l_2)\in \mathbb{Z}^2\\ (l_1l_2,p)=1\\ 2\alpha_1l_2\alpha\equiv l_1(\alpha_1-\alpha_2\alpha^2) \bmod{p}}} \hat{\Phi}\left(\frac{l_1N}{p^{n-r}}\right)\hat{\Phi}\left(\frac{l_2N}{p^{n-r}}\right) S_{\alpha}\left(f_{p^rl_1,p^rl_2};p^n\right)+O_{\varepsilon}(1)
\end{equation*}
if $N\ge p^{n\varepsilon}$. 

Let $D:=\alpha_1\alpha_2l_1^2+\alpha_1^2l_2^2$. The congruence \eqref{keycong} has a double root $\alpha \bmod{p}$ iff $D\equiv 0 \bmod{p}$, and in this case we get $\alpha^2\equiv -\alpha_1\overline{\alpha_2}\bmod{p}$ which is excluded from the summation over $\alpha$. Hence, only the case $D\not\equiv 0\bmod{p}$ occurs in which we have no root if $D$ is a quadratic non-residue modulo $p$ and two roots of multiplicity one if $D$ is a quadratic residue modulo $p$. Therefore, we may assume from now on that $D\not\equiv 0 \bmod{p}$ and $D$ is a quadratic residue modulo $p$. Then using Proposition \ref{Expsums}, if $\alpha$ satisfies \eqref{keycong}, we obtain
$$
S_{\alpha}\left(f_{p^rl_1,p^rl_2};p^n\right)=
\begin{cases}
e_{p^n}\left(f_{p^rl_1,p^rl_2}(\alpha^{\ast})\right)\cdot p^{(n+r)/2} & \mbox{ if } n-r \mbox{ is even,}\\
e_{p^n}\left(f_{p^rl_1,p^rl_2}(\alpha^{\ast})\right)\cdot \left(\frac{A(\alpha)}{p}\right)\cdot \frac{G_p}{\sqrt{p}}\cdot p^{(n+r)/2} & \mbox{ if } n-r \mbox{ is odd,}
\end{cases}
$$
where $\alpha^{\ast}$ is the unique lifting of $\alpha$ to a root of the congruence
$$
2\alpha_1l_2\alpha^{\ast}\equiv l_1(\alpha_1-\alpha_2(\alpha^{\ast})^2) \bmod{p^{n-r}}
$$
and 
$$
A(\alpha)=\frac{2 f_{p^rl_1,p^rl_2}''(\alpha)}{p^r}.
$$
We calculate that 
$$
\alpha^{\ast}\equiv (-\alpha_1l_2\pm\sqrt{D})\overline{\alpha_2l_1} \bmod{p^{n-r}},
$$
where $\sqrt{D}$ denotes one of the two roots of the congruence
$$
x^2\equiv D\bmod{p^{n-r}}.
$$
A short calculation gives
$$
e_{p^n}\left(f_{p^rl_1,p^rl_2}(\alpha^{\ast})\right)=e_{p^{n-r}}\left(\pm bx_3\overline{\alpha_1} \sqrt{D}\right).
$$
Further, we calculate the second derivative of $f_{p^rl_1,p^rl_2}$ to be
$$
f_{p^rl_1,p^rl_2}''(t)=4\alpha_1\alpha_2bx_3\cdot \frac{-2\alpha_2l_1t-l_2(\alpha_1-\alpha_2t^2)}{(\alpha_1+\alpha_2t^2)^3}\cdot p^{r}.
$$
Another short calculation gives 
$$
A(\alpha)=-\frac{2bx_3(\alpha_2l_1)^2}{\alpha_2\alpha^2\sqrt{D}}
$$
and hence 
$$
\left(\frac{A(\alpha)}{p}\right)=\left(\frac{-2bx_3\alpha_2\sqrt{D}}{p}\right). 
$$

As noted above, the cases $\alpha^2\not\equiv 0,\pm \alpha_1\overline{\alpha_2}\bmod{p}$ cannot occur if $\alpha$ is a multiple root of the congruence \eqref{keycong} with $(l_1l_2,p)=1$. So altogether, we obtain
\begin{equation} \label{Uaftereva}
\begin{split}
U= & \frac{N^2}{p^{3n/2}}\sum\limits_{r=0}^{n-2} p^{r/2} \sum\limits_{\substack{(l_1l_2,p)=1\\ D=\Box\bmod{p}}} \hat{\Phi}\left(\frac{l_1N}{p^{n-r}}\right)\hat{\Phi}\left(\frac{l_2N}{p^{n-r}}\right)\times \\ & \sum\limits_{(x_3,p)=1} \Phi\left(\frac{x_3}{N}\right)\cdot C_{n-r}(x_3,D)\cdot 
\left(e_{p^{n-r}}\left(bx_3\overline{\alpha_1}\sqrt{D}\right)+ e_{p^{n-r}}\left(-bx_3\overline{\alpha_1}\sqrt{D}\right)\right)+O_{\varepsilon}(1),
\end{split}
\end{equation}
where $D=\Box\bmod{p}$ means that $D$ is a quadratic residue modulo $p$ and 
$$
C_{n-r}(x_3,D):=\begin{cases} 1 & \mbox{ if } n-r \mbox{ is even,}\\ 
\left(\frac{-2bx_3\alpha_2\sqrt{D}}{p}\right)\cdot \frac{G_p}{\sqrt{p}} & \mbox { if } n-r
\mbox{ is odd.} \end{cases}
$$
If $(x_3D,p)=1$, then
$$
C_{n-r}(x_3,D)= \frac{G_{p^{n-r}}}{p^{(n-r)/2}}\cdot 
\left(\frac{-2bx_3\alpha_2\sqrt{D}}{p^{n-r}}\right)
$$
in each of the two cases above. Therefore, $U$ can be more compactly written as 
\begin{equation} \label{compact}
\begin{split}
U= & \frac{N^2}{p^{3n/2}}\sum\limits_{r=0}^{n-2} p^{r/2} \cdot \frac{G_{p^{n-r}}}{p^{(n-r)/2}} \cdot \sum\limits_{\substack{D=1\\ D\equiv \Box \bmod{p}}}^{\infty} F_{n-r}(D) \times\\
 & \sum\limits_{x_3\in \mathbb{Z}} \Phi\left(\frac{x_3}{N}\right)\cdot \left(\frac{x_3}{p^{n-r}}\right) \cdot 
\left(e_{p^{n-r}}\left(bx_3\overline{\alpha_1}\sqrt{D}\right)+ e_{p^{n-r}}\left(-bx_3\overline{\alpha_1}\sqrt{D}\right)\right)+O_{\varepsilon}(1),
\end{split}
\end{equation}
where 
\begin{equation} \label{FD}
F_{n-r}(D):= \left(\frac{-2b\alpha_2\sqrt{D}}{p^{n-r}}\right)\cdot \sum\limits_{\substack{(l_1l_2,p)=1\\ \alpha_1\alpha_2l_1^2+\alpha_1^2l_2^2=D}} \hat{\Phi}\left(\frac{l_1N}{p^{n-r}}\right)\hat{\Phi}\left(\frac{l_2N}{p^{n-r}}\right).
\end{equation}

\subsection{Single Poisson summation and final count} 
Now we split the sum over $x_3$ in \eqref{compact} into sub-sums over residue classes modulo $p$ and perform Poisson summation, getting
\begin{equation*}
\begin{split}
\sum\limits_{x_3\in \mathbb{N}} \Phi\left(\frac{x_3}{N}\right)\cdot \left(\frac{x_3}{p^{n-r}}\right) \cdot e_{p^{n-r}}\left(\pm bx_3\overline{\alpha_1}\sqrt{D}\right) = 
& \sum\limits_{u=1}^{p} \left(\frac{u}{p^{n-r}}\right)
\sum\limits_{x_3\equiv u\bmod{p}} \Phi\left(\frac{x_3}{N}\right)\cdot
e_{p^{n-r}}\left(b x_3\overline{\alpha_1}\sqrt{D}\right)\\
= & \frac{N}{p} \sum\limits_{v\in \mathbb{Z}} \left(\sum\limits_{u=1}^{p} \left(\frac{u}{p^{n-r}}\right) \cdot e_{p}(uv)\right) 
\cdot \hat\Phi\left(\frac{N}{p}\left(\frac{\pm b\overline{\alpha_1}\sqrt{D}}{p^{n-r-1}}-v\right)\right).
\end{split}
\end{equation*}
Using the rapid decay of $\hat\Phi$, the above is $O(N)$ if $||b\overline{\alpha_1}\sqrt{D}/p^{n-r-1}||\le p^{1+n\varepsilon}N^{-1}$ and negligible otherwise, provided $n$ is large enough. We may constraint $b\overline{\alpha_1}\sqrt{D}$ to the range $0\le \sqrt{D}\le p^{n-r}$ and then write $b\overline{\alpha_1}\sqrt{D}=wp^{n-r-1}+l_3$, where $w=0,...,p-1$, $(l_3,p)=1$ and $|l_3|\le L_r$ with
$$
L_r:=p^{n-r+n\varepsilon}N^{-1}.
$$ 
Moreover, the summations over $l_1$ and $l_2$ in \eqref{FD} can be cut off at $|l_1|,|l_2|\le L_r$ at the cost of a negligible error if $n$ is large enough. It follows that
 \begin{equation*}
\begin{split}
U\ll & \frac{N^3}{p^{3n/2}}\sum\limits_{r=0}^{n-2} p^{r/2} \sum\limits_{w=0}^{p-1} \sum\limits_{\substack{(l_1,l_2,l_3)\in \mathbb{Z}^3\\ (l_1l_2l_3,p)=1\\ |l_1|,|l_2|,|l_3|\le L_r\\ b^2\overline{\alpha_1}^{2}(\alpha_1\alpha_2l_1^2+\alpha_1^{2}l_2^2)\equiv (wp^{n-r-1}+l_3)^2\bmod{p^{n-r}}}} 1 +O_{\varepsilon}(1).
\end{split}
\end{equation*}
Recalling that $b^2\equiv -\alpha_3\overline{\alpha_2} \bmod{p^{n-r}}$, the congruence above implies 
\begin{equation} \label{acongru}
\alpha_2\alpha_3l_1^2+\alpha_1\alpha_3l_2^2+\alpha_1\alpha_2l_3^2\equiv 0\bmod{p^{n-r-1}},
\end{equation}
and hence 
 \begin{equation} \label{UboundI} 
\begin{split}
U\ll & \frac{N^3}{p^{3n/2}}\sum\limits_{r=0}^{n-2} p^{r/2}  \sum\limits_{\substack{(l_1,l_2,l_3)\in \mathbb{Z}^3\\ (l_1l_2l_3,p)=1\\ |l_1|,|l_2|,|l_3|\le L_r\\ \alpha_2\alpha_3l_1^2+\alpha_1\alpha_3l_2^2+\alpha_1\alpha_2l_3^2\equiv 0\bmod{p^{n-r-1}}}} 1 +O_{\varepsilon}(1).
\end{split}
\end{equation}
Let $H:=\max\{|\alpha_1|,|\alpha_2|,|\alpha_3|\}$.
Now if $3H^2L_r^2< p^{n-r-1}$, then the congruence above can be replaced by the equation 
$$
\alpha_2\alpha_3l_1^2+\alpha_1\alpha_3l_2^2+\alpha_1\alpha_2l_3^2=0.
$$ 
Certainly, this is the case if $N\ge p^{n/2+2n\varepsilon}$ and $n$ is large enough. (At this place, we use our assumption that $\alpha_1,\alpha_2,\alpha_3$ are {\it fixed}!) Hence, in this case, we have
\begin{equation*} 
\begin{split}
U\ll & \frac{N^3}{p^{3n/2}}\sum\limits_{r=0}^{n-2} p^{r/2} \sum\limits_{\substack{(l_1,l_2,l_3)\in \mathbb{Z}^3\\ (l_1l_2l_3,p)=1\\ |l_1|,|l_2|,|l_3|\le L_r\\ \alpha_2\alpha_3l_1^2+\alpha_1\alpha_3l_2^2+\alpha_1\alpha_2l_3^2=0}} 1 +O_{\varepsilon}(1).
\end{split}
\end{equation*}

Now we apply Proposition \ref{quadequations} with $A=\alpha_2\alpha_3$, $B=\alpha_1\alpha_3$, $C=-\alpha_1\alpha_2l_3^2$, $X=l_1$ and $Y=l_2$ to bound $U$ by 
\begin{equation*} 
\begin{split}
U\ll & \frac{N^3}{p^{3n/2}}\sum\limits_{r=0}^{n-2} p^{r/2} L_r^{1+\varepsilon} \ll \frac{N^2}{p^{n/2}}\cdot p^{9n\varepsilon}.
\end{split}
\end{equation*}
This needs to be compared to the main term which is of size
$$
T_0\asymp \frac{N^3}{p^n}.
$$
If $N\ge p^{(1/2+10\varepsilon)n}$, then $U=o(T_0)$, which completes the proof of Theorem \ref{mainresult} in this case. 
 
\subsection{Parametrization of solutions - Cases II} \label{2ndpara}
Now we consider the case when none of $-\alpha_i\alpha_j$ is a quadratic residue modulo $p$. By Proposition \ref{para}, the $\mathbb{Q}_p$-rational points $(y_1,y_2)$ on the conic 
$$
\alpha_1y_1^2+\alpha_2y_2^2=-\alpha_3,
$$
are parametrized as  
\begin{equation*}
\begin{split}
y_1=y_1(t)&:=a-2\alpha_2\frac{at^2-bt}{\alpha_1+\alpha_2t^2},\\
y_2=y_2(t)&:=-b-2\alpha_1\frac{at-b}{\alpha_1+\alpha_2t^2},
\end{split}
\end{equation*}
where $\alpha_1a^2+\alpha_2b^2=-\alpha_3$ (in particular, $\alpha_1a^2+\alpha_2b^2\equiv -\alpha_3\bmod{p^n}$). This equation is soluble in $(a,b)$ as a consequence of .... 
We define 
$$
M_0:=\{(y_1(t),y_2(t)) : t=1,2,...,p^n\},
$$
$$
M_s:=\left\{\left(y_1\left(\frac{t}{p^s}\right),y_2\left(\frac{t}{p^s}\right)\right):t=1,...p^{n-s},t\not\equiv0~~\text{mod}~~ p\right\} \mbox{ for } s=1,2...,n
$$
and 
$$
M:=\bigcup\limits_{s=0}^n M_s.
$$
Noting that $y_1(t/p^s)$ and $y_2(t/p^s)$ are $p$-adic integers, we will view $y_1(t/p^s)$ and $y_2(t/p^s)$ as elements of $\mathbb{Z}/p^n\mathbb{Z}$. 

It can be seen that the pairs $(y_1(t/p^s),y_2(t/p^s)$ in the above sets $M_s$ ($s=0,...,n$) are distinct and $M_i\cap M_j=\emptyset$ for $i\neq j$ by the following argument.  
Suppose that $$
\left(y_1\left(\frac{t_1}{p^s_1}\right),y_2\left(\frac{t_1}{p^s_1}\right)\right)= \left(y_1\left(\frac{t_2}{p^s_2}\right),y_2\left(\frac{t_2}{p^s_2}\right)\right),
$$
which implies 
\begin{equation*}
\begin{split}
a-2\alpha_2\frac{at_1^2-bt_1p^{s_1}}{\alpha_1p^{2s_1}+\alpha_2t_1^2}=&a-2\alpha_2\frac{at_2^2-bt_2p^{s_2}}{\alpha_1p^{2s_2}+\alpha_2t_2^2},\\
-b-2\alpha_1p^{s_1}\frac{at_1-bp^{s_1}}{\alpha_1p^{2s_1}+\alpha_2t_1^2}=&-b-2\alpha_1p^{s_2}\frac{at_2-bp^{s_2}}{\alpha_1p^{2s_1}+\alpha_2t_2^2}.
\end{split}
\end{equation*}
Then a short calculation gives
\begin{equation} \label{ok} t_2p^{s_1}=t_1p^{s_2}\bmod{p^n}.
\end{equation}
So if $0\leq s_1=s_2\leq n$ then $p^{s_1}(t_2-t_1)\equiv 0\bmod{p^n}$, which implies $t_2-t_1\equiv0\bmod{p^{n-s_1}}$. Hence $t_2=t_1$ because $1\leq t_1,t_2\leq p^{n-s_1}$.
If $s_2>s_1$ then we have $t_2\equiv 0 \bmod {p^{s_2-s_1}}$, which contradicts the fact that $t_2\not\equiv 0\bmod{p}$.

Therefore we get that $|M_0|=p^n,|M_n|=1$, $|M_s|=p^{n-s}-p^{n-s-1}$ for $s=1,2...n-1$, and $$|M|=\left|\bigcup_{s=0}^nM_s\right|=p^n+(p^{n-1}-p^{n-2})+....+(p^2-p)+(p-1)+1=p^n+p^{n-1}.$$
Now Proposition \ref{nofroot} tells us that $M=\bigcup_{s=0}^nM_n$ is a complete set of solutions of $\alpha_1y_1^2+\alpha_2y_2^2\equiv-\alpha_3$ mod $p^n$.
We also observe that if $(x_1,x_2)$ is a solution of $\alpha_1x_1^2+\alpha_2x_2^2\equiv-\alpha_3$ mod $p^n$ then $(x_1x_2,p)=1$ because of the fact that $-\alpha_1\alpha_2$ is a quadratic non-residue.

\subsection{Double Poisson summation} As in \eqref{fdob} we write
\begin{equation*}
T= \frac{N^2}{p^{2n}}\sum\limits_{(x_3,p)=1} \Phi\left(\frac{x_3}{N}\right)\sum\limits_{(k_1,k_2)\in \mathbb{Z}^2} \hat{\Phi}\left(\frac{k_1N}{p^n}\right)\hat{\Phi}\left(\frac{k_2N}{p^n}\right)\sum\limits_{\substack{y_1,y_2\bmod{p^n}\\ (y_1y_2,p)=1 \\ \alpha_1y_1^2+\alpha_2y_2^2+\alpha_3\equiv 0\bmod{p^n}}} e_{p^n}\left(k_1x_3y_1+k_2x_3y_2\right).
\end{equation*}
Using the above parametrization, we deduce that 
\begin{equation*}
T= \frac{N^2}{p^{2n}}\sum\limits_{(x_3,p)=1} \Phi\left(\frac{x_3}{N}\right)\sum\limits_{(k_1,k_2)\in \mathbb{Z}^2} \hat{\Phi}\left(\frac{k_1N}{p^n}\right)\hat{\Phi}\left(\frac{k_2N}{p^n}\right)\sum_{s=0}^n\sum\limits_{(y_1,y_2)\in M_s} e_{p^n}\left(k_1x_3y_1+k_2x_3y_2\right).
\end{equation*}
We decompose $T$ into 
\begin{equation*} \label{2divide}
T=T_0+U,
\end{equation*}
where $T_0$ is the main term contribution of $(k_1,k_2)=(0,0)$. Hence,
\begin{equation*}
\begin{split}
T_0&= \hat{\Phi}(0)^2\cdot \frac{N^2}{p^{2n}}\sum\limits_{(x_3,p)=1} \Phi\left(\frac{x_3}{N}\right) \cdot\left|M\right|\\
&= \hat{\Phi}(0)^2\cdot \frac{N^2}{p^{2n}}\sum\limits_{(x_3,p)=1} \Phi\left(\frac{x_3}{N}\right) \cdot(p^n+p^{n-1})\\
 &=\hat{\Phi}(0)^2\cdot \frac{N^2}{p^{2n}}\sum\limits_{(x_3,p)=1} \Phi\left(\frac{x_3}{N}\right) \cdot p^{n-1}(p-s_p(\alpha_1,\alpha_2,\alpha_3)),
\end{split}
\end{equation*}
where $s_p(\alpha_1,\alpha_2,\alpha_3))$ is defined as in \eqref{sdef}. Now if $N\geq p^{n\varepsilon}$ for any fixed $\varepsilon>0$, the term $T_0$ can be further simplified as in \eqref{fdob}, and we obtain
\begin{equation*} \label{T0II}
\begin{split}
T_0= \hat{\Phi}(0)^3 \cdot \frac{(p-s_p(\alpha_1,\alpha_2,\alpha_3))(p-1)}{p^2}\cdot \frac{N^3}{p^{n}}\cdot\left(1+o(1)\right)= \hat{\Phi}(0)^3\cdot C_p(\alpha_1,\alpha_2,\alpha_3)\cdot \frac{N^3}{p^n}\cdot \left(1+o(1)\right). 
\end{split}
\end{equation*}

\subsection{Evaluation of exponential sums} Now we look at the error contribution
\begin{equation} \label{2errorcont}
U= \frac{N^2}{p^{2n}}\sum\limits_{(x_3,p)=1} \Phi\left(\frac{x_3}{N}\right)\sum\limits_{(k_1,k_2)\in \mathbb{Z}^2\setminus \{(0,0)\}} \hat{\Phi}\left(\frac{k_1N}{p^n}\right)\hat{\Phi}\left(\frac{k_2N}{p^n}\right) E\left(k_1,k_2,x_3;p^n\right)
\end{equation}
with 
\begin{equation*}
\begin{split}
E\left(k_1,k_2,x_3;p^n\right):=\sum_{s=0}^n\sum\limits_{(y_1,y_2)\in M_s} e_{p^n}\left(x_3(k_1y_1+k_2y_2)\right).
\end{split}
\end{equation*}
Assume that
$$
(k_1,k_2,p^n)=p^r.
$$
The contribution of $r=n-1,n$ to the right-hand side of \eqref{2errorcont} is $O_{\varepsilon}(1)$ if $N\ge p^{n\varepsilon}$ by the rapid decay of $\hat\Phi$ since $(k_1,k_2)=(0,0)$ is excluded from the summation. In the following, we assume that $r\le n-2$ so that Proposition \ref{Expsums} is applicable.

Let $$f_{s,k_1,k_2}(t)=x_3\left(k_1y_1\left(\frac{t}{p^s}\right)+k_2y_2\left(\frac{t}{p^s}\right)\right),$$ where $y_1(t),y_2(t)$ is defined in \eqref{2ndpara}.
We have $f_{s,k_1,k_2}(t)\equiv f_{s,k_1,k_2}(t+wt^{n-s})\bmod{p^n}$ for $w=1,2...,p^{s}$
and  deduce that 
\begin{equation}\label{newsum}
\begin{split}
E\left(k_1,k_2,x_3;p^n\right):=\sum_{t=0}^{p^n}e_{p^n}\left(f_{0,k_1,k_2}(t)\right)+\sum_{s=1}^n\frac{1}{p^s}\sum\limits_{\substack{t=1\\t\not\equiv 0~~\text{mod}~~p}}^{p^n} e_{p^n}\left( f_{s,k_1,k_2}(t)\right).
\end{split}
\end{equation}
The derivative of the amplitude function turns out to be
\begin{equation*}
\begin{split}
f'_{s,k_1,k_2}(t)&=2x_3\frac{k_1\alpha_2(-b\alpha_2t^2p^{2s}-2a\alpha_1tp^{3s}+\alpha_1bp^{4s})+k_2\alpha_1(a\alpha_2t^2p^{2s}-2b\alpha_2tp^{3s}-a\alpha_1p^{4s})}{(\alpha_1p^{2s}+\alpha_2t^2)^2 }\\
&=2x_3\frac{\alpha_2(k_2\alpha_1a-k_1\alpha_2b)t^2p^{2s}-2\alpha_1\alpha_2(ak_1+bk_2)tp^{3s}-\alpha_1(k_2\alpha_1a-k_1\alpha_2b)p^{4s}}{(\alpha_1p^{2s}+\alpha_2t^2)^2}.
\end{split}
\end{equation*}
If $s>0$ then $t\not\equiv 0$ mod $p$, which implies $\alpha_1p^{2s}+\alpha_2t^2\not\equiv 0 \bmod{p}$. If $s=0$ then also $\alpha_1+\alpha_2t\not\equiv 0\bmod{p}$ because of the fact that $-\alpha_1\alpha_2$ is a quadratic non-residue modulo $p$.

It is easy to see that $\left(k_1,k_2,p^n\right)=\left((k_2\alpha_1a-k_1\alpha_2 b),(a k_1+b k_2),p^n\right)$ and therefore $\mbox{ord}_p(f_{s,k_1,k_2}')=p^{2s+r}$.
We split the second sum over $t$ on the right-hand side of \eqref{newsum} into
\begin{equation*}
\begin{split}
\sum\limits_{\substack{t=1\\t\not\equiv 0~~\text{mod}~~p}}^{p^n} e_{p^n}\left( f_{s,k_1,k_2}(t)\right)=\sum\limits_{\substack{\alpha=1\\ \alpha \not\equiv 0 \bmod{p}}}^p S_{\alpha}\left(f_{s,k_1,k_2};p^n\right).
\end{split}
\end{equation*}
We see that $s>0$ and $p^{-r-2s}f'_{s,k_1,k_2}(t)\equiv 0\bmod{p}$ imply $t\equiv0\bmod{p}$. Then using Proposition \ref{Expsums}, we have $S_{\alpha}=0$ if $\alpha\neq0$, which implies that 
\begin{equation*}
 \sum_{s=1}^n\frac{1}{p^s}\sum\limits_{\substack{t=1\\t\not\equiv 0~~\text{mod}~~p}}^{p^n} e_{p^n}\left( f_{s,k_1,k_2}(t)\right)=0.
\end{equation*}
It follows that
\begin{equation*} \label{2nderrorcont}
U= \frac{N^2}{p^{2n}}\sum\limits_{(x_3,p)=1} \Phi\left(\frac{x_3}{N}\right)\sum\limits_{(k_1,k_2)\in \mathbb{Z}^2\setminus \{(0,0)\}} \hat{\Phi}\left(\frac{k_1N}{p^n}\right)\hat{\Phi}\left(\frac{k_2N}{p^n}\right) \sum\limits_{\substack{t=1}}^{p^n} e_{p^n}\left( f_{0,k_1,k_2}(t)\right).
\end{equation*}
We split the inner-most sum over $t$ into 
\begin{equation*} \label{2split}
\sum\limits_{\substack{t=1}}^{p^n} e_{p^n}\left( f_{0,k_1,k_2}(t)\right)=\sum\limits_{\substack{\alpha=1}}^p S_{\alpha}\left(f_{0,k_1,k_2};p^n\right),
\end{equation*}
where 
\begin{equation*} \label{2Salphadef}
S_{\alpha}\left(f_{0,k_1,k_2};p^n\right)=\sum\limits_{\substack{t \bmod{p^n}\\ t\equiv \alpha\bmod{p}}} e_{p^n}\left(f_{0,k_1,k_2}(t)\right).
\end{equation*}
Set
$$
l_1:=\frac{k_1}{p^r}, \quad l_2:=\frac{k_2}{p^r}.
$$
If $(k_1,p^n)>(k_2,p^n)$ then 
$$p^{-r}f'_{0,k_1,K_2}\equiv 0\bmod{p},$$
which implies $$a\alpha_2t^2-2b\alpha_2t-a\alpha_1\equiv 0 \bmod{p}.$$
The above congruence relation has no solution because the determinant of the corresponding polynomial equals $-4\alpha_2\alpha_3$ and is therefore a quadratic non-residue. So by Proposition \ref{Expsums},
\begin{equation*}
 \sum\limits_{\substack{t=1}}^{p^n} e_{p^n}\left( f_{0,k_1,k_2}(t)\right)=0. 
\end{equation*}
Similarly, this sum is zero if $(k_1,p^n)>(k_2,p^n)$.
Thus it follows that $(l_1l_2,p)=1$ and $(k_1,p^n)=(k_2,p^n)=p^r$.
Using Proposition \ref{Expsums}, we have $S_{\alpha}(f_{0,k_1,k_2})=0$ unless
\begin{equation}\label{2keycong}
 l_2\alpha_1(a\alpha_2\alpha^2-2b\alpha_2\alpha-a\alpha_1)\equiv l_1\alpha_2(b\alpha_2\alpha^2+2a\alpha_1\alpha-\alpha_1b)\bmod{p}.
 \end{equation}
Set $$C(t):=l_2\alpha_1(a\alpha_2t^2-2b\alpha_2t-a\alpha_1)- l_1\alpha_2(b\alpha_2t^2+2a\alpha_1t-\alpha_1b).$$
In summary, we have 
\begin{equation*}
U= \frac{N^2}{p^{2n}}\sum\limits_{(x_3,p)=1} \Phi\left(\frac{x_3}{N}\right)\sum\limits_{r=0}^{n-2} \sum\limits_{\substack{\alpha=1}}^p \sum\limits_{\substack{(l_1,l_2)\in \mathbb{Z}^2\\ (l_1l_2,p)=1\\ C(\alpha)\equiv 0\bmod{p} }} \hat{\Phi}\left(\frac{l_1N}{p^{n-r}}\right)\hat{\Phi}\left(\frac{l_2N}{p^{n-r}}\right) S_{\alpha}\left(f_{0,p^rl_1,p^rl_2};p^n\right)+O_{\varepsilon}(1)
\end{equation*}
if $N\ge p^{n\varepsilon}$.

Let 
$$
D:=-\alpha_3\overline{\alpha_2}(\alpha_1\alpha_2l_1^2+\alpha_1^2l_2^2) \bmod{p^n}.
$$ 
The congruence \eqref{2keycong} has a double root $\alpha \bmod{p}$ iff $D\equiv 0 \bmod{p}$, and in this case we get $\alpha^2\equiv -\alpha_1\overline{\alpha_2}\bmod{p}$ which contradicts the fact that $-\alpha_1\overline{\alpha_2}$ is a quadratic non-residue. Hence, only the case $D\not\equiv 0\bmod{p}$ occurs in which we have no root if $D$ is a quadratic non-residue modulo $p$ and two roots of multiplicity one if $D$ is a quadratic residue modulo $p$. Therefore, we may assume from now on that $D\not\equiv 0 \bmod{p}$ and $D$ is a quadratic residue modulo $p$. Then using Proposition \ref{Expsums}, if $\alpha$ satisfies \eqref{2keycong}, we obtain
$$
S_{\alpha}\left(f_{p^rl_1,p^rl_2},p^n\right)=
\begin{cases}
e_{p^n}\left(f_{0,p^rl_1,p^rl_2}(\alpha^{\ast})\right)\cdot p^{(n+r)/2} & \mbox{ if } n-r \mbox{ is even,}\\
e_{p^n}\left(f_{0,p^rl_1,p^rl_2}(\alpha^{\ast})\right)\cdot \left(\frac{A(\alpha)}{p}\right)\cdot \frac{G_p}{\sqrt{p}}\cdot p^{(n+r)/2} & \mbox{ if } n-r \mbox{ is odd,}
\end{cases}
$$
where $\alpha^{\ast}$ is the unique lifting of $\alpha$ to a root of the congruence
\begin{equation}
 l_2\alpha_1(a\alpha_2(\alpha^{\ast})^2-2b\alpha_2\alpha^{\ast}-a\alpha_1)\equiv l_1\alpha_2(b\alpha_2(\alpha^{\ast})^2+2a\alpha_1\alpha^{\ast}-\alpha_1b)\bmod{p^{n-r}}
 \end{equation}
and 
$$
A(\alpha)=\frac{2 f_{0,p^rl_1,p^rl_2}''(\alpha)}{p^r}.
$$
We calculate that 
$$
\alpha^{\ast}\equiv \frac{\alpha_1(al_1+bl_2)\pm\sqrt{D}}{l_2\alpha_1a-l_1\alpha_2b} \bmod{p^{n-r}},
$$
where $\sqrt{D}$ denotes one of the two roots of the congruence
$$
x^2\equiv D\bmod{p^{n-r}}.
$$
A short calculation gives
$$
e_{p^n}\left(f_{0,p^rl_1,p^rl_2}(\alpha^{\ast})\right)=e_{p^{n-r}}\left(\pm x_3\overline{\alpha_1}\cdot \sqrt{D}\right).
$$
Further, we calculate the second derivative of $f_{p^rl_1,p^rl_2}$ to be
$$ 
f_{0,p^rl_1,p^rl_2}''(t)=4\alpha_1\alpha_2x_3\cdot \frac{\alpha_2D_2t^2+2D_1t-\alpha_1D_2}{(\alpha_1+\alpha_2t^2)^3}\cdot p^{r},
$$
where $D_1=l_1\alpha_1a-l_1\alpha_2b$ and $D_2=l_1a+l_2b$.
Another short calculation gives 
$$
A(\alpha)=\frac{2x_3D_1^2}{\alpha_2\alpha^2\sqrt{D}}
$$
and hence 
$$
\left(\frac{A(\alpha)}{p}\right)=\left(\frac{2x_3\alpha_2\sqrt{D}}{p}\right). 
$$

So altogether, we obtain
\begin{equation*} \label{2Uaftereva}
\begin{split}
U= & \frac{N^2}{p^{3n/2}}\sum\limits_{r=0}^{n-2} p^{r/2} \sum\limits_{\substack{(l_1l_2,p)=1\\ D=\Box\bmod{p}}} \hat{\Phi}\left(\frac{l_1N}{p^{n-r}}\right)\hat{\Phi}\left(\frac{l_2N}{p^{n-r}}\right)\times \\ & \sum\limits_{(x_3,p)=1} \Phi\left(\frac{x_3}{N}\right)\cdot C_{n-r}(x_3,D)\cdot 
\left(e_{p^{n-r}}\left(x_3\overline{\alpha_1}\sqrt{D}\right)+ e_{p^{n-r}}\left(-x_3\overline{\alpha_1}\sqrt{D}\right)\right)+O_{\varepsilon}(1),
\end{split}
\end{equation*}
where $D=\Box\bmod{p}$ means that $D$ is a quadratic residue modulo $p$ and 
$$
C_{n-r}(x_3,D):=\begin{cases} 1 & \mbox{ if } n-r \mbox{ is even,}\\ 
\left(\frac{2x_3\alpha_2\sqrt{D}}{p}\right)\cdot \frac{G_p}{\sqrt{p}} & \mbox { if } n-r
\mbox{ is odd.} \end{cases}
$$
If $(x_3D,p)=1$, then
$$
C_{n-r}(x_3,D)= \frac{G_{p^{n-r}}}{p^{(n-r)/2}}\cdot 
\left(\frac{2x_3\alpha_2\sqrt{D}}{p^{n-r}}\right)
$$
in each of the two cases above. Therefore, $U$ can be more compactly written as 
\begin{equation*} \label{2compact}
\begin{split}
U= & \frac{N^2}{p^{3n/2}}\sum\limits_{r=0}^{n-2} p^{r/2} \cdot \frac{G_{p^{n-r}}}{p^{(n-r)/2}} \cdot \sum\limits_{\substack{D=1\\ D\equiv \Box \bmod{p}}}^{\infty} F_{n-r}(D) \times\\
 & \sum\limits_{x_3\in \mathbb{Z}} \Phi\left(\frac{x_3}{N}\right)\cdot \left(\frac{x_3}{p^{n-r}}\right) \cdot 
\left(e_{p^{n-r}}\left(x_3\overline{\alpha_1}\sqrt{D}\right)+ e_{p^{n-r}}\left(-x_3\overline{\alpha_1}\sqrt{D}\right)\right)+O_{\varepsilon}(1),
\end{split}
\end{equation*}
where 
\begin{equation*} \label{2FD}
F_{n-r}(D):= \left(\frac{2\alpha_2\sqrt{D}}{p^{n-r}}\right)\cdot \sum\limits_{\substack{(l_1l_2,p)=1\\ -\alpha_3\overline{\alpha_2}(\alpha_1\alpha_2l_1^2+\alpha_1^2l_2^2)=D }} \hat{\Phi}\left(\frac{l_1N}{p^{n-r}}\right)\hat{\Phi}\left(\frac{l_2N}{p^{n-r}}\right).
\end{equation*}
This should be compared to \eqref{compact}. Similar calculations as in subsection 3.4. now lead to precisely the same bound \eqref{UboundI}, and the rest of the proof is then the same as in Case I. This completes the proof of Theorem \ref{mainresult} in Case II. 

\section{Proof of Theorem \ref{mainresult2}}
Let $p>5$ be a prime and $q=p^n$.
We denote the quantity in question as 
$$
\Sigma_{\alpha_1,\alpha_2,\alpha_3}(\Phi,N,q)=\sum\limits_{\substack{(x_1,x_2,x_3)\in \mathbb{Z}^3\\ (x_1x_2x_3,p)=1\\ \alpha_1x_1^2+\alpha_2x_2^2+\alpha_3x_3^2 \equiv 0 \bmod{q}}} \Phi\left(\frac{x_1}{N}\right)
\Phi\left(\frac{x_2}{N}\right)\Phi\left(\frac{x_3}{N}\right).
$$ 
In particular, letting $\chi_{[-1,1]}$ be the characteristic function of the interval $[-1,1]$, we have 
$$
\Sigma_{\alpha_1,\alpha_2,\alpha_3}(\chi_{[-1,1]},N,q)=\sum\limits_{\substack{|x_1|,|x_2|,|x_3|\le N\\ \alpha_1x_1^2+\alpha_2x_2^2+\alpha_3x_3^3\equiv 0 \bmod{q}\\ (x_1x_2x_3,q)=1}} 1,
$$
which we denote just by $\Sigma_{\alpha_1,\alpha_2,\alpha_3}(N,q)$ throughout the sequel. 

If $\Phi$ is a Schwartz class function, then Theorem \ref{mainresult} yields, for {\it fixed} $\alpha_1,\alpha_2,\alpha_3$ and $n\rightarrow \infty$, an asymptotic formula for $\Sigma_{\alpha_1,\alpha_2,\alpha_3}(\Phi,N,q)$ if $N\gg q^{1/2+\varepsilon}$. 
Now we allow $\alpha_1,\alpha_2,\alpha_3$ to be {\it arbitrary} (i.e., to vary with $n$). In this situation, we will see that using our approach, it is relatively easy to get an asymptotic formula if $N\gg q^{2/3+\varepsilon}$. We will describe how to reach the exponent $2/3$ and then refine our method to beat it. 

Combining \eqref{divide}, \eqref{T0I}, \eqref{UboundI} in Case I and the corresponding equations and inequalities in Case II, we get an asymptotic formula of the form 
\begin{equation} \label{essential}
\Sigma_{\alpha_1,\alpha_2,\alpha_3}(\Phi,N,q)=\Phi(0)^3\cdot C_p(\alpha_1,\alpha_2,\alpha_3)\cdot \frac{N^3}{q} \cdot (1+o(1))+O\left(\frac{N^3}{q^{3/2}}\sum\limits_{r=0}^{n-2} p^{r/2} \Sigma_{\beta_1,\beta_2,\beta_3}(L_r,q_r) +O_{\varepsilon}(1)\right),
\end{equation}
where 
$$
L_r:=p^{-r}q^{1+\varepsilon}N^{-1}, \quad q_r:=p^{-r-1}q,\quad \beta_1:=\alpha_2\alpha_3, \quad \beta_2:=\alpha_1\alpha_3,\quad \beta_3:=\alpha_1\alpha_2. 
$$
Obviously, $\beta_1,\beta_2,\beta_3$ are also coprime to the modulus. 

Now our task becomes to bound from above the quantity
$\Sigma_{\beta_1,\beta_2,\beta_3}(M,q')$, where $M=L_r$ and $q'=q_r$. We need an upper bound for this quantity in the situation when $M\le (q')^{1/2-\varepsilon}$.
If $\alpha_1,\alpha_2,\alpha_3$ and hence $\beta_1,\beta_2,\beta_3$ are {\it fixed}, then by our method in the previous section, we easily get
\begin{equation} \label{firstbound}
\Sigma_{\beta_1,\beta_2,\beta_3}(M,q')\ll M^{1+\varepsilon}
\end{equation}
in this situation, where the implied $\ll$-constant depends on $\alpha_1,\alpha_2,\alpha_3$. Hence, we then obtain
$$
\Sigma_{\alpha_1,\alpha_2,\alpha_3}(\Phi,N,q)=\hat\Phi(0)^3\cdot C_p(\alpha_1,\alpha_2,\alpha_3)\cdot \frac{N^3}{q}+O\left(\frac{N^2}{q^{1/2-\varepsilon}}\right),
$$  
which gives an asymptotic if $N\gg q^{1/2+\varepsilon}$. We may conjecture that the bound \eqref{firstbound} holds with an {\it absolute} $\ll$-constant, but the dependence on $\beta_1,\beta_2,\beta_2$ (and hence $\alpha_1,\alpha_2,\alpha_3$) seems to be difficult to remove. What we can establish instead relatively easily, with an absolute $O$-constant, is 
 \begin{equation} \label{secondbound}
\Sigma_{\beta_1,\beta_2,\beta_3}(M,q')\ll M^{3/2+\varepsilon},
\end{equation}
again provided that $M\le (q')^{1/2-\varepsilon}$. 
This implies then
$$
\Sigma_{\alpha_1,\alpha_2,\alpha_3}(N,q)=\hat\Phi(0)^3\cdot C_p(\alpha_1,\alpha_2,\alpha_3)\cdot \frac{N^3}{q}+O\left(N^{3/2}q^{\varepsilon}\right),
$$  
which yields an asymptotic for {\it arbitrary} $\alpha_1,\alpha_2,\alpha_3$ if $N\gg q^{2/3+\varepsilon}$. 

There are several ways to establish \eqref{secondbound}. Below we describe an elementary method which has space for improvements. We then refine this method to beat the exponent $3/2$ in \eqref{secondbound}. For easy of notation, we write $q$ in place of $q'$ in the following, bearing in mind that it is not the original modulus $q$. We also assume $M\le q^{1/2-\varepsilon}$ henceforth. Set 
$$
\mathcal{M}:=\left\{-\beta_3x_3^2 : |x_3|\le M,\ (x_3,q)=1\right\}. 
$$  
Then 
$$
\Sigma_{\beta_1,\beta_2,\beta_3}(M,q)=2\sum\limits_{m\in \mathcal{M}} \sum\limits_{\substack{|x_1|,|x_2|\le M\\ \beta_1x_1^2+\beta_2x_2^2\equiv m \bmod{q}\\ (x_1x_2,q)=1}} 1.
$$
Now we apply the Cauchy-Schwarz inequality, getting
\begin{equation*}
\begin{split}
\Sigma_{\beta_1,\beta_2,\beta_3}(M,q)^2\le & |\mathcal{M}| \sum\limits_{m\in \mathcal{M}} \left| \sum\limits_{\substack{|x_1|,|x_2|\le M\\ \beta_1x_1^2+\beta_2x_2^2\equiv m \bmod{q}\\ (x_1x_2,q)=1}} 1\right|^2\\
= & M \sum\limits_{m\in \mathcal{M}} \sum\limits_{\substack{|x_1|,|x_2|,|y_1|,|y_2|\le M\\ \beta_1x_1^2+\beta_2x_2^2\equiv m \bmod{q}\\
\beta_1y_1^2+\beta_2y_2^2\equiv m \bmod{q}\\
 (x_1x_2y_1y_2,q)=1}} 1\\
= & M(\mathcal{D}+\mathcal{E}), 
\end{split}
\end{equation*}
where 
$$
\mathcal{D}:=\sum\limits_{m\in \mathcal{M}} \sum\limits_{\substack{|x_1|,|x_2|\le M\\ \beta_1x_1^2+\beta_2x_2^2\equiv m \bmod{q}\\
|y_1|=|x_1| \ \mbox{\scriptsize and } |y_2|=|x_2| \\
 (x_1x_2,q)=1}} 1 = 4\Sigma_{\beta_1,\beta_2,\beta_3}(M,q)
$$
and 
$$
\mathcal{E}:=\sum\limits_{m\in \mathcal{M}} \sum\limits_{\substack{|x_1|,|x_2|,|y_1|,|y_2|\le M\\ \beta_1x_1^2+\beta_2x_2^2\equiv m \bmod{q}\\
\beta_1y_1^2+\beta_2y_2^2\equiv m \bmod{q}\\|y_1|\not=|x_1| \ \mbox{\scriptsize or } |y_2|\not=|x_2|\\
 (x_1x_2y_1y_2,q)=1}} 1.
$$
We observe that 
\begin{equation*}
\begin{split}
\mathcal{E}\le & \sum\limits_{\substack{|x_1|,|x_2|,|y_1|,|y_2|\le M\\ \beta_1x_1^2+\beta_2x_2^2\equiv \beta_1y_1^2+\beta_2y_2^2 \bmod{q}\\|y_1|\not=|x_1| \ \mbox{\scriptsize or } |y_2|\not=|x_2|}} 1\\
 = & \sum\limits_{\substack{|x_1|,|x_2|,|y_1|,|y_2|\le M\\ \beta_1(x_1-y_1)(x_1+y_1)\equiv \beta_2(y_2-x_2)(y_2+x_2) \bmod{q}\\ (x_1-y_1)(x_1+y_1)\not=0 \ \mbox{\scriptsize or } (y_2-x_2)(y_2+x_2)\not=0}} 1.
\end{split}
\end{equation*}
We further observe that if one of the numbers 
$$
A_1=(x_1-y_1)(x_1+y_1) \quad \mbox{and} \quad A_2=(y_2-x_2)(y_2+x_2)
$$
in the summation condition above equals 0, then the other one equals 0 as well. To see this, note that if $\beta_i A_i\equiv 0 \bmod{q}$, then $p^s|(x_i-y_i)$ and $p^t|(x_i+y_i)$ with $s+t\ge n$, which is not possible if $M\le q^{1/2-\varepsilon}$ unless $x_i-y_i=0$ or $x_i+y_i=0$. 
It follows that
\begin{equation*}
\begin{split}
\mathcal{E}\ll \sum\limits_{\substack{|x_1|,|x_2|,|y_1|,|y_2|\le M\\ \beta_1(x_1-y_1)(x_1+y_1)\equiv \beta_2(y_2-x_2)(y_2+x_2) \bmod{q}\\ (x_1-y_1)(x_1+y_1)\not=0 \ \mbox{\scriptsize and } (y_2-x_2)(y_2+x_2)\not=0}} 1\\ 
 \le \sum\limits_{\substack{0<|A_1|,|A_2|\le 2M^2\\ \beta_1A_1\equiv \beta_2A_2\bmod{q}}} \tau(|A_1|)\tau(|A_2|),
\end{split}
\end{equation*}
where $\tau(k)$ denotes the number of divisors of $k\in \mathbb{N}$. Since we know that $\tau(k)\ll k^{\varepsilon}$, we deduce that
$$
\mathcal{E}\ll M^{\varepsilon}F_{\beta_1,\beta_2}(2M^2,q),
$$
where
$$
F_{\beta_1,\beta_2}(X,q):=\sum\limits_{\substack{0<|A_1|,|A_2|\le X\\ \beta_1A_1\equiv \beta_2A_2\bmod{q}}} 1.
$$

We observe that if $M\le q^{1/2-\varepsilon}$ and $q$ is large enough, then any given $A_2$ in the summation condition above fixes $A_1$, if it exists at all. Hence, we trivially get
\begin{equation} \label{trivial}
F_{\beta_1,\beta_2}(2M^2,q)\le 4M^2. 
\end{equation}
Collecting everything above, we arrive at
$$
\Sigma_{\beta_1,\beta_2,\beta_3}(M,q)^2\ll M\left(\Sigma_{\beta_1,\beta_2,\beta_3}(M,q)
+M^{2+\varepsilon}\right),
$$
implying the claimed bound
$$
\Sigma_{\beta_1,\beta_2,\beta_3}(M,q)\ll M^{3/2+\varepsilon}.
$$
The bound \eqref{trivial} is sharp: If $\beta_1=\beta_2$, then necessarily $A_1=A_2$ and we get exactly
$$
F_{\beta_1,\beta_2}(2M^2,q)= 4M^2. 
$$
However, for {\it generic} $\beta_1$ and $\beta_2$, we should expect a much better bound. Indeed, if $\beta_1\overline{\beta_2}/q$ satisfies certain Diophantine properties, then we can get a saving over the trivial bound. In \cite[equation (63)]{BD}, we established that
$$
F_{\beta_1,\beta_2}(X,q)\ll \left(\frac{X^2}{q}+\frac{rX}{q}+\frac{X}{r}+1\right)(rXq)^{\varepsilon},
$$
provided that
$$
\frac{\beta_1\overline{\beta_2}}{q}=\frac{a}{r}+O(r^{-2})
$$
with $a\in \mathbb{Z}$, $r\in \mathbb{N}$ and $(a,r)=1$. Again combining everything above, we obtain
\begin{equation*}
\begin{split}
\Sigma_{\beta_1,\beta_2,\beta_3}(M,q)^2\ll & M\left(\Sigma_{\beta_1,\beta_2,\beta_3}(M,q)+M^{\varepsilon}F_{\beta_1,\beta_2}(2M^2,q)\right)\\
\ll & M\left(\Sigma_{\beta_1,\beta_2,\beta_3}(M,q)+\left(\frac{M^4}{q}+\frac{rM^2}{q}+\frac{M^2}{r}+1\right)(rMq)^{2\varepsilon}\right),
\end{split}
\end{equation*}
which implies the bound
\begin{equation} \label{method1}
\Sigma_{\beta_1,\beta_2,\beta_3}(M,q)\ll \left(\frac{M^{5/2}}{q^{1/2}}+\frac{r^{1/2}M^{3/2}}{q^{1/2}}+\frac{M^{3/2}}{r^{1/2}}+M\right)(rMq)^{\varepsilon}.
\end{equation}
So if $r$ is not too small or too large, we may get a saving over the trivial bound $\ll M^{3/2+\varepsilon}$. 

The same arguments as above can be applied for $\beta_1\overline{\beta_3}/q$ or $\beta_2\overline{\beta_3}/q$ in place of $\beta_1\overline{\beta_2}/q$. Now for $i=1,2$, Dirichlet's approximation theorem tells us that if $Q\in \mathbb{N}$, then there exists $r_i\le Q$ and $a_i\in \mathbb{Z}$ with $(a_i,r_i)=1$ such that 
\begin{equation}\label{appro}
\left|\frac{\beta_i\overline{\beta_3}}{q}-\frac{a_i}{r_i}\right|\le \frac{1}{r_iQ}\le \frac{1}{r_i^2}.
\end{equation}
Suppose $R\in \mathbb{N}$ is another parameter. Both, $R$ and $Q$ will be fixed later. If one of $r_1$ and $r_2$ exceeds $R$, then we deduce that
\begin{equation} \label{Method1}
\Sigma_{\beta_1,\beta_2,\beta_3}(M,q)\ll \left(\frac{M^{5/2}}{q^{1/2}}+\frac{Q^{1/2}M^{3/2}}{q^{1/2}}+\frac{M^{3/2}}{R^{1/2}}+M\right)(QMq)^{\varepsilon}.
\end{equation}
If both $r_1$ and $r_2$ are less than $R$, then we will use the following different method. We first write our congruence in question as 
$$
x_3^2\equiv -\beta_1\overline{\beta_3}x_1^2-\beta_2\overline{\beta_3}x_2^2\bmod{q}.
$$ 
Multiplying by $r:=r_1r_2$ gives
\begin{equation} \label{newcong}
rx_3^2\equiv -\beta_1\overline{\beta_3}rx_1^2-\beta_2\overline{\beta_3}rx_2^2\bmod{q}.
\end{equation}
By \eqref{appro}, we have
$$
-\beta_1\overline{\beta_3}r=-a_1r_2q + O\left(\frac{qr_2}{Q}\right) 
$$
and 
$$
-\beta_2\overline{\beta_3}r=-a_2r_1q+O\left(\frac{qr_1}{Q}\right).
$$
Now reducing the right-hand side of \eqref{newcong} modulo $q$, we deduce that
\begin{equation} \label{newcong1}
rx_3^2\equiv \gamma_1x_1^2+\gamma_2x_2^2\bmod{q},
\end{equation}
where 
$$
\gamma_1:=-\beta_1\overline{\beta_3}r+a_1r_2q
$$
and 
$$
\gamma_2:=-\beta_2\overline{\beta_3}r+a_2r_1q,
$$
and $\gamma_1$, $\gamma_2$ satisfy the bound
$$
\gamma_1,\gamma_2\ll \frac{q(r_1+r_2)}{Q}\ll \frac{qR}{Q}.
$$
We also have $r\le R^2$.
Now we write \eqref{newcong1} as an equation in the form
\begin{equation} \label{equ}
rx_3^2+kq=\gamma_1x_1^2+\gamma_2x_2^2,
\end{equation}
where 
$$
k=O\left(\frac{R^2M^2}{q}+\frac{RM^2}{Q}\right)
$$
since $|x_1|,|x_2|,|x_3|\le M$. If the left-hand side of \eqref{equ} is fixed, then by Proposition \ref{quadequations} there are $O((RM)^{\varepsilon})$ solutions $(x_2,x_3)$ with $|x_2|,|x_3|\le M$. 
Since we have 
$$
O\left(\frac{R^2M^3}{q}+\frac{RM^3}{Q}+M\right)
$$
choices for the left-hand side of \eqref{equ}, it follows that 
\begin{equation} \label{Method2}
\Sigma_{\beta_1,\beta_2,\beta_3}(M,q)\ll \left(\frac{R^2M^3}{q}+\frac{RM^3}{Q}+M\right)(RM)^{\varepsilon}.
\end{equation}
Now using \eqref{Method1} and \eqref{Method2} in the relevant complementary cases, we obtain
\begin{equation*}
\Sigma_{\beta_1,\beta_2,\beta_3}(M,q)\ll \left(\frac{M^{5/2}}{q^{1/2}}+\frac{Q^{1/2}M^{3/2}}{q^{1/2}}+\frac{M^{3/2}}{R^{1/2}}+\frac{R^2M^3}{q}+\frac{RM^3}{Q}+M\right)(QRMq)^{\varepsilon}.
\end{equation*}
Choosing 
$$
R:=\lceil q^{2/5}M^{-3/5} \rceil \quad \mbox{and} \quad Q:=\lceil q^{3/5}M^{3/5} \rceil 
$$
and recalling $M\le q^{1/2-\varepsilon}$ gives
$$
\Sigma_{\beta_1,\beta_2,\beta_3}(M,q)\ll \left(\frac{M^{5/2}}{q^{1/2}}+\frac{M^{9/5}}{q^{1/5}}+M\right)q^{\varepsilon}.
$$
It follows that
$$
p^{r/2}\Sigma_{\beta_1,\beta_2,\beta_3}\left(L_r,q_r\right)\ll \left(\frac{q^2}{N^{5/2}}+\frac{q^{8/5}}{N^{9/5}}+\frac{q}{N}\right)q^{\varepsilon}
$$
for $r=0,...,n-2$. 
Plugging this into \eqref{essential} gives
$$
\Sigma_{\alpha_1,\alpha_2,\alpha_3}(N,q)=\hat\Phi(0)^3\cdot C_p(\alpha_1,\alpha_2,\alpha_3)\cdot \frac{N^3}{q}+O\left(\left(N^{1/2}q^{1/2}+N^{6/5}q^{1/10}+\frac{N^2}{q^{1/2}}\right)q^{\varepsilon}\right).
$$
The $O$-term is smaller than the main term if 
$$
N\ge q^{11/18+\varepsilon}.
$$
So if this inequality is satisfied, we get an asymptotic. This completes the proof of Theorem \ref{mainresult2}.

\end{document}